\documentclass[final]{siamltex}

\usepackage{amsmath,amssymb}
\usepackage{booktabs}
\usepackage{graphicx}
\usepackage[
colorlinks=true,
citecolor=blue,
urlcolor=blue,
linktocpage,
pagebackref
]{hyperref}
\newcommand{\ju}{\tilde{i}}
\newcommand{\eu}{\tilde{f}}
\newcommand{\jn}{\breve{i}}
\newcommand{\en}{\breve{f}}
\newcommand{\hK}{\hat K}
\newcommand{\hx}{\hat{x}}
\renewcommand{\diag}{\dof{diag}}
\newcommand{\dB}{\dof{B}}
\newcommand{\dC}{\dof{C}}
\newcommand{\dR}{\dof{R}}
\newcommand{\dI}{\dof{I}}
\newcommand{\dPi}{\dof{\Pi}}
\newcommand{\dD}{\dof{D}}
\newcommand{\dH}{\dof{H}}
\newcommand{\dS}{\dof{S}}
\newcommand{\dT}{\dof{T}}
\newcommand{\dP}{\dof{P}}
\newcommand{\dv}{\dof{v}}
\newcommand{\du}{\dof{u}}
\newcommand{\dw}{\dof{w}}
\newcommand{\dr}{\dof{r}}
\newcommand{\dx}{\dof{x}}
\newcommand{\dy}{\dof{y}}
\newcommand{\hsig}{\hat \sigma}
\newcommand{\hsign}{\hat\sigma_n}
\newcommand{\bsig}{\bar{\sigma}_n}
\newcommand{\hq}{\hat{q}}
\newcommand{\hG}{\hat {\mathcal{G}}}
\newcommand{\hF}{\hat {\mathcal{F}}}
\newcommand{\FF}{ {\mathcal{F}}}

\newcommand{\hE}{\hat {\mathcal{E}}}
\newcommand{\RRR}{\mathbb{R}}

\newcommand{\domain}{\Omega}
\newcommand{\ran}{\mathop\mathrm{ran}}
\newcommand{\diam}{\mathop\mathrm{diam}}
\newcommand{\mesh}{\domain_h}

\newcommand{\normal}{n}
\newcommand{\om}{\Omega}
\newcommand{\oh}{\mesh}
\renewcommand{\d}{\partial}
\newcommand{\dof}[1]{{\tt{#1}}}
\newcommand{\mx}{\dof{x}}
\newcommand{\mz}{\dof{z}}
\newcommand{\mv}{\dof{v}}
\newcommand{\mB}{\dof{B}}
\newcommand{\mM}{\dof{M}}

\newcommand{\trialspace}{U_h} 
\newcommand{\interfacespace}{Q_h}
\newcommand{\testspace}{Y_h} 

\newcommand{\lagrange}{\trialspace}
\newcommand{\nedelec}{N_h}
\newcommand{\rt}{R_h} 

\renewcommand{\div}{\operatorname{div}}
\newcommand{\dive}{\mathrm{div}}
\newcommand{\curl}{\operatorname{curl}}

\newcommand{\Hodiv}[1]{\mathring{H}(\div,{#1})}
\newcommand{\Hdiv}{H(\div,\domain)}

\newcommand{\hdiv}[1]{H(\div,{#1})}

\newcommand{\extension}{\mathcal E^{\div}}
\newcommand{\trace}{\operatorname{trc}_{\normal}}
\newcommand{\extmat}{\dof{E}}

\newcommand{\curlmatrix}{\dof{C}}

\newcommand{\projmatrix}{\dof{\Pi}}
\newcommand{\adsprec}{\dof{B}}

\newcommand{\coeff}{\kappa}

\def\zz{\phantom{0}}
\def\zzz{\phantom{000}}

\title{A scalable preconditioner for a DPG method%
\thanks{Performed under the auspices of the U.S. Department of Energy under
   Contract DE-AC52-07NA27344 (LLNL-JRNL-710378) and supported in part by AFOSR
   grant FA9550-12-1-0484.}}

\author{
A.~T.~Barker\thanks{
   Center for Applied Scientific Computing (CASC),
   Lawrence Livermore National Laboratory, Livermore, CA 94550,
   {\tt atb@llnl.gov}, {\tt dobrev1@llnl.gov}, {\tt tzanio@llnl.gov}}
\and V.~Dobrev\footnotemark[2]
\and J.~Gopalakrishnan\thanks{
   Portland State University, PO Box 751 (MTH), Portland, OR 97207-0751,
   {\tt gjay@pdx.edu}}
\and T.~Kolev\footnotemark[2]
}

\begin{document}
\maketitle

\begin{abstract}
  We show how a scalable preconditioner for the primal discontinuous
  Petrov-Galerkin (DPG) method can be developed using existing
  algebraic multigrid (AMG) preconditioning techniques. The stability
  of the DPG method gives a norm equivalence which allows us to
  exploit existing AMG algorithms and software.  We show how these
  algebraic preconditioners can be applied directly to a Schur
  complement system of interface unknowns arising from the DPG method.
  To the best of our knowledge, this is the first massively scalable
  algebraic preconditioner for DPG problems.
\end{abstract}



\pagestyle{myheadings}
\thispagestyle{plain}
\markboth{}{}

\section{Introduction} \label{sec:intro}

Discontinuous Petrov-Galerkin (DPG) methods, introduced
in~\cite{DG:2010:DPGI,DG:2011:DPGII}, constructed test spaces that
guarantee stability. Today these methods are known to be
simultaneously viewable as Galerkin mixed methods, as least-squares
methods in  nonstandard norms, or as Petrov-Galerkin methods using
discontinuous functions~\cite{DemkoGopal16}.  DPG methods have  a great
deal of flexibility, allowing them to be applied to a wide variety of problems
\cite{ChanHeuerBui-T14,DG:2011:DPGIV,DGN:2011:DPGIII}, and their
convergence theory has now matured~\cite{CarstDemkoGopal16}.

However, there is a lack of fast scalable solvers for the DPG method.
In~\cite{barker-brenner-one-level-dpg}, an overlapping Schwarz
preconditioner is analyzed: because it has no coarse level, the
preconditioner expectedly deteriorates as overlap size become small.
A coarse level was added for improved scalability
in~\cite{li-xu-dd-dpg}, where the authors analyzed a two-level
additive Schwarz preconditioner for an ultraweak DPG method applied to
the Poisson equation with Robin boundary condition. Going beyond the
Poisson problem to the harmonic wave equation, there are numerical
reports of good performance of certain preconditioning
strategies~\cite{GopalSchob15}.  To our knowledge, no multilevel
preconditioners have been proposed for the DPG method until very
recently in~\cite{RoberChan16}, where geometric multilevel strategies
are investigated numerically but without theoretical analysis.  In
this paper we show how existing algebraic multilevel preconditioners
can be effectively combined to precondition a DPG system, including at
very large scale on parallel supercomputers.

The particular DPG method we consider is the so-called primal DPG
method~\cite{demkowicz-gopalakrishnan-primal-dpg}, reviewed in the
next section. After describing a basic norm equivalence associated
with the method, we proceed to analyze one of the component norms in
Section~\ref{sec:qnorm}. We show that this interface norm, obtained as
an infimum over an infinite-dimensional space, is equivalent to an
infimum over a finite-dimensional space.  An auxiliary algebraic Schur
complement result is presented in Section~\ref{sec:schur-result}.
Section~\ref{sec:ads} identifies the finite-dimensional infimum as a
Schur complement norm and proceeds to analyze an auxiliary-space
preconditioner for the Schur complement. The preconditioner and the
main result are summarized in
Section~\ref{sec:prec}. Section~\ref{sec:results} reports results from
numerical studies of the proposed preconditioner.  We conclude by
summarizing the main results from the paper in
Section~\ref{sec:conclusions}.

\section{The primal DPG method} \label{sec:dpg}

For completeness and consistency of notation, we recall some
definitions and results from the work introducing the primal DPG
system~\cite{demkowicz-gopalakrishnan-primal-dpg}. The model problem
we consider is the Poisson problem of finding $ u \in H_0^1(\domain) $
such that
\begin{equation}
  \int_\om \coeff \nabla u \cdot \nabla v\; dx = \int_\om f v \; dx
  \label{eq:poisson}
\end{equation}
for all $ v \in H_0^1(\domain) $, where $ \domain $ is a polygonal (in
$ \mathbb R^2 $) or polyhedral domain (in $ \mathbb R^3$) with
Lipschitz boundary and $ \coeff > 0 $ is a piecewise constant coefficient.
Zero Dirichlet boundary conditions on
$ \partial \domain $ are essentially imposed in~\eqref{eq:poisson}.
In practice the method can also handle more general problems with
varying coefficients and different boundary conditions, but we choose
this setting for simplicity.

Even before discretization, the DPG formulation uses a mesh-dependent
weak form.  We assume that $\domain$ is given together with a mesh
$\mesh$ that partitions $\domain$ into elements of varying shapes. For
example, in the $ \mathbb R^2$ case, the mesh elements $K \in \mesh $
may be triangles or quadrilaterals, while in $ \mathbb R^3$
the elements may be tetrahedra, prisms, hexahedra, etc. Precise
assumptions on the mesh and element shapes will be specified later,
but for now we only require that the boundary of each element be
Lipschitz, so that traces of Sobolev space functions on the element
boundaries are well-defined. Specifically, we require a well-defined
normal trace operator
\[
\trace : \Hdiv   \to \prod_{K \in \mesh } H^{-1/2}(\partial K),
\]
which maps to element-wise traces
$(\trace q) |_{\d K} = q \cdot n |_{\d K}$.  One can also
define $\trace q$ as a single valued function on the mesh facets:
Indeed, if each
interface facet $\gamma = \bar{K}_1 \cap \bar{K}_2$ between two
elements $K_1,K_2 \in \oh$, as well as boundary facets
$\gamma = \bar{K}_1 \cap \d\om$ are Lipschitz, then we may fix a
continuous unit normal vector function $n_\gamma$ on $\gamma$ and
define $(\trace q )|_\gamma = n_\gamma \cdot q|_\gamma.$ This is a
well-defined function in the dual space of
$\mathring{H}^{1/2}(\gamma)$ (denoted also by
$H^{1/2}_{00}(\gamma)$), whenever $q \in \Hdiv$.

The numerical fluxes of the DPG method lie in the range of $\trace$,
i.e., in the space $ Q = \ran (\trace) $ with norm given by
\begin{align}
  \label{eq:Qnorm}
  \| q \|_Q &= \inf_{ \tau \in \trace^{-1}\{ q\} } \|\tau \|_{\Hdiv}.
\end{align}
Here, as usual, $\trace^{-1}\{ q\}$ denotes the pre-image of the
singleton $\{ q \}$.  It is standard to prove that the {\em minimal
extension operator} $E : Q \to \Hdiv,$ defined by $\trace (Eq) = q$
and $(E q, v)_{\hdiv K} = 0$ for all $v \in \Hodiv K$ and
$K \in \oh,$ attains the infimum at~\eqref{eq:Qnorm}, i.e.,
\begin{equation}
  \label{eq:Eq}
  \| q \|_Q = \| Eq \|_{\Hdiv}.
\end{equation}
Here and throughout, for any inner product space $W$, we use
$\| \cdot \|_W$ and $(\cdot, \cdot)_W$ to denote its norm and inner
product, respectively. When the space is uniquely understood from the
argument or other context, we will drop the subscript.  Let
\[
H^1(\oh) = \prod_{K \in \oh} H^1(K).
\]
This product space is endowed with the standard Cartesian product norm and inner
product.  For brevity, put $X = H_0^1(\om) \times Q$ and
$Y = H^1(\oh)$.  Define the bilinear form
$b : X \times Y \to \mathbb R$ by
\[
b( (w, r), v) =
\sum_{K \in \oh}
\left(
\int_K \nabla w \cdot \nabla v\; dx
+ \langle r, v \rangle_{H^{-1/2}(\d K)}
\right)
\]
where $\langle r, v \rangle_{H^{-1/2}(\d K)}$ denotes the duality
pairing between $H^{-1/2}(\d K)$ and $H^{1/2}(\d K)$.  The
Dirichlet problem~\eqref{eq:poisson} can then be
reformulated~\cite{demkowicz-gopalakrishnan-primal-dpg} as the problem
of finding a pair $(u, q) \in X$ satisfying
\begin{equation}
  \label{eq:dpgweak}
  b( (u,q), v) = F(v), \qquad \forall v \in Y,
\end{equation}
with $F(v) = (f,v)_{L^2(\om)}$.  It is proved
in~\cite[Lemma 3.4]{demkowicz-gopalakrishnan-primal-dpg} that
the problem~\eqref{eq:dpgweak} is uniquely solvable for $(u,q) \in X$
given any $F$ in the dual space of $Y$ (see also
\cite[Example~3.6]{CarstDemkoGopal16} for a simplified analysis).

The primal DPG method uses the formulation~\eqref{eq:dpgweak} and
finite element subspaces $U_h \subset H_0^1(\om)$, $Q_h \subset Q$,
and $Y_h \subset H^1(\oh)$. Here $h = \max_{K \in \oh} \diam(K)$ denotes the mesh size of $\oh$. To
describe a computational version of the method, let
$\{u_i\}, \{ q_j\}, \{v_k\}$ denote a finite element basis for $U_h$,
$Q_h$ and $Y_h$, respectively.  Define the matrices
\[
[\dof{B}_0]_{ki} = b( (u_i, 0), v_k), \quad
[\dof{B}_1]_{kj} = b( (0, q_j), v_k), \quad
[\dof{M}]_{kl} = (v_l,v_k)_{H^1(\oh)}
\]
and  set
\[
\dof B =
\begin{bmatrix}
  \dof {B_0} & \dof {B_1}
\end{bmatrix},
\qquad
\dof{A}= \dof{B}^T \dof{M}^{-1} \dof B =
\begin{bmatrix}
  \dof{B}_0^T \dof{M}^{-1} \dof{B_0} & \dof{B}_0^T \dof{M}^{-1} \dof{B}_1
  \\
  \dof{B}_1^T \dof{M}^{-1} \dof{B_0} & \dof{B}_1^T \dof{M}^{-1} \dof{B}_1
\end{bmatrix}.
\]
Let $\dof u$ denote the vector in $\mathbb{R}^{\dim(U_h)}$
representing a function $u \in U_h$ by the basis expansion formula
$u = \sum_i [\dof u]_i u_i.$ The vectors $\dof q$ in
$\mathbb{R}^{\dim(Q_h)}$ and $\dof v$ in $\mathbb{R}^{\dim(Y_h)}$ are
similarly defined. Restricting~\eqref{eq:dpgweak} to the finite dimensional
spaces formally gives
\begin{equation*}
  \dof B
  \begin{bmatrix}
    \dof u \\ \dof q
  \end{bmatrix}
  =
  \dof F
\end{equation*}
where $[\dof F]_k = F(v_k)$.
The DPG discretization of~\eqref{eq:dpgweak} solves instead the following
symmetric and positive definite problem for $\dof u$ and $\dof q$:
\begin{equation}
  \label{eq:dpgmat}
  \dof A
  \begin{bmatrix}
    \dof u \\ \dof q
  \end{bmatrix}
  =
  \dof g
\end{equation}
where $\dof g = \dof{B}^T \dof{M}^{-1} \dof{F}$.  Note that $ \dof M$ is the Gram matrix
in the ``broken'' $H^1(\oh)$-inner product and so is block diagonal
(one block per element).  Thus, $ \dof{M}^{-1} $ can be evaluated fast
locally.

The DPG method admits three well-known interpretations.  The early
papers on the DPG method used the concept of {\em optimal test
functions}~\cite{DG:2010:DPGI}. Its interpretation as a least-squares
method in a nonstandard inner product was pointed out
in~\cite[p.~6]{DG:2011:DPGII}. Its interpretation as a mixed method
is now well known (see e.g.,~\cite[Theorem~2.4]{BoumaGopalHarb14}).
It is easy to see that all these three interpretations, in practice,
yield the same matrix system~\eqref{eq:dpgmat} when the same spaces and
bases are used.

The starting point of our analysis is the stability of the DPG
method~\eqref{eq:dpgmat}.  Let $X_h = U_h \times Q_h$ and let $\mx$
and $\mz$ be vectors representing two functions $x$ and $z$ in $X_h$,
respectively. Per the above-mentioned notational conventions, $(\dof A
\dof x, \mz)$ denotes the Euclidean inner product $\mz^T \dof{A} \dof
x$. It defines a bilinear form in the function space~$X_h$, namely
$a(x,z) = (\dof A \dof x, \dof z)$. Note that both $X_h$ and $Y_h$ are
used in the definition of $\dof A$. Throughout this paper we assume
that the mesh $\oh$ and the spaces $X_h$ and $Y_h$ are such that there
exist mesh-independent constants $ c_1 $ and $ c_ 2 $ satisfying
\begin{equation}
  c_1 \| x \|_X^2
  \leq
  a(x,x)
  \leq
  c_2 \| x \|_X^2,
  \label{eq:q-equivalence}
\end{equation}
for all $x \in X_h$. The connection between~\eqref{eq:q-equivalence}
and the stability of the method is described next.

\begin{proposition}
  Assumption~\eqref{eq:q-equivalence} holds
  if and only if
  \begin{equation}
    \label{eq:bsup}
  c_1 \| x \|_X \le \sup_{ 0 \ne v \in Y_h}
  \frac{ |b(x, v)| }{ \| v\|_Y } \le c_2 \| x \|_X
  \end{equation}
  for all $x$ in $X_h$.
\end{proposition}
\begin{proof}
  Define $T_h : X_h \to Y_h$ by $ (T_h x, y)_Y = b(x, y), $ for all
  $x \in X_h$ and $y \in Y_h$. Then, for any $x \in X_h$,
  \[
    \| T_h x \|_X^2 =
    \sup_{ 0 \ne v \in Y_h}
  \frac{ (T_h x, v)_Y }{ \| v\|_Y } =
    \sup_{ 0 \ne v \in Y_h}
  \frac{ b(x, v) }{ \| v\|_Y }.
  \]
  Letting $\mx$ and $\mv$ denote the vector representations of
  $x \in X_h$ and $v = T_h x \in Y_h$, respectively, it is easy to see
  that $\mv = \mM^{-1} \mB \mx$. Hence the result follows from
  $ a(x,x) = (\mM \mM^{-1} \mB \mx, \mM^{-1} \mB \mx) = \| T_h
  x\|_X^2.$
\end{proof}

Clearly, the upper inequality of~\eqref{eq:bsup} follows from the
continuity of the bilinear form $b(\cdot, \cdot),$ and therefore holds
independently of the choice of the discrete spaces. The lower
inequality of~\eqref{eq:bsup} is an inf-sup condition. It follows from
the fact that~\eqref{eq:dpgweak} is well-posed whenever the discrete spaces
are chosen so that a Fortin operator~\cite{GopalQiu14} can be
constructed. Here are a few known examples of cases where a $c_1$
independent of~$h$ can be obtained for the Dirichlet problem under
consideration:

\bigskip

\begin{enumerate}
\item Suppose the mesh $\oh$ is a quasiuniform tetrahedral
  geometrically conforming mesh, $U_h$ is the Lagrange finite element
  space of degree~$p$, $Q_h = \{ q: q|_\gamma $ is a polynomial of
  degree at most ${p-1}$ on each mesh facet $\gamma\}$, and
  $Y_h = \{ v : v|_K $ is a polynomial of degree at most $p+2\}.$ Then
  a Fortin operator provided in~\cite{GopalQiu14} yields a $c_1$
  independent of~$h$, as proved in~\cite{demkowicz-gopalakrishnan-primal-dpg}.
\item When $\oh$ is a uniform mesh of rectangular elements,
  $U_h = \{ w \in H_0^1(\om): w|_K $ is in the tensor product space of
  polynomials of degree at most $p$ in each coordinate direction, for
  all elements $K\in \oh\}$, $Q_h = \{ q: q|_\gamma \in P_{p}(\gamma)$
  on each mesh facet $\gamma\}$, and $Y_h= \{ v : v|_K $ is a
  polynomial of degree at most $p+3$ in each coordinate direction$\},$ a Fortin operator
  in~\cite{CaloColliNiemi14} gives a mesh-independent $c_1$.
\end{enumerate}

\bigskip

In the remainder of this paper, we examine an important implication
of~\eqref{eq:q-equivalence}. Namely, in order to precondition the large
Hermitian positive definite DPG system~\eqref{eq:dpgmat}, it suffices
to obtain preconditioner for the $\| \cdot \|_X$ norm. In  our
model problem, this norm is
\[
\| x \|_X^2 = \| u \|_{H^1(\om)}^2 + \| q \|_Q^2.
\]
for any $x = (u,q) \in U_h \times Q_h,$ so
it suffices to combine preconditioners for the
$H^1(\om)$ and $Q$ norms. Since the former is standard, we focus on the
latter in the next section.

\section{Characterizing  the $Q$-norm} \label{sec:qnorm}

The $Q$-norm~\eqref{eq:Qnorm} is defined through a minimization over
an infinite dimensional space (the minimal extension~$E$
in~\eqref{eq:Eq} is not computable). In this section, we relate this norm
to a minimum over a finite dimensional subspace.

\subsection{Tetrahedral case}   \label{ssec:tet}

To present the idea transparently, we first detail the case when $\oh$
is a geometrically conforming mesh of tetrahedral elements.  For any
tetrahedron $K$, let $R_p(K) = P_p(K)^3 + x P_p(K),$ where $x$ is the
coordinate vector and $P_p(K)$ denotes the set of all polynomials of
total degree at most~$p \geq 1$. The Raviart-Thomas finite element space is
$R_h = \{ r \in \Hdiv : r|_K \in R_p(K)$ for all $K \in \oh\}$. Let
$Q_h = \trace(R_h).$ Clearly $Q_h$ is a finite dimensional subspace
of~$Q$. Define $E_h : Q_h \to R_h$ by $\trace (E_hq) = q,$ and
\[
(E_h q, v)_{\hdiv K} = 0, \qquad\forall  v \in R_p(K) \cap \Hodiv{K},
\]
for all $K \in \oh.$ This computable approximation of the minimal
extension operator defines a new norm on $Q_h$,
\[
\| q \|_{Q_h} = \| E_h q \|_{\Hdiv}.
\]

We now proceed to prove the equivalence of this norm with the $Q$-norm
(in Theorem~\ref{thm:tet} below).  Throughout this section, let $c$
denote a generic positive constant whose value might change from one
occurrence to another, but will remain {\em independent of $h$ and
$p$}.  Let $\hK$ denote the unit tetrahedron, $\hat n$ denote its
outward unit normal on $\d \hK$, and
$\hsign = \hsig \cdot \hat n|_{\d \hK}$.  Let
$Q_p(\d \hK) = \{ \hsign: \; \hsig \in R_p(\hK)\}.$ For any
$\hsig \in R_p(\hK)$ define the constant function
\[
\bsig  =  \frac{1}{| \d \hK |} \int_{\d \hK} \hsign\; ds \,,
\]
where $| \d \hK|$ denotes the surface area of $\d \hK.$

\begin{lemma}
  \label{lem:G}
  There is a $c>0$ and a $\hG : \RRR \to R_p(\hK)$ such that for
  any $\hsig$ in  $\hdiv \hK$ with $\hsign \in Q_p(\d \hK)$,
  we have
  $\hat n \cdot (\hG \bsig)|_{\d \hK}  = \bsig$,
  \[
  \| \hG \bsig \|_{L^2(\hK)} \le c \| \hsig \|_{\hdiv \hK},
  \quad \text{ and } \quad
  \| \div \hG \bsig \|_{L^2(\hK)}
  \le c \| \div \hsig \|_{\hK}.
  \]
  In the case of tetrahedral elements, we actually prove the stronger inequality
  \[
  \| \hG \bsig \|_{\hdiv \hK} \le c \| \div \hsig \|_{\hK}.
  \]
\end{lemma}
\begin{proof}
  Let $\hx_I$ denote the incenter of the unit tetrahedron $\hK$. Then
  $ (\hx - \hx_I) \cdot \hat{n} = d $ is constant for any
  $ \hx \in \partial \hK$  ($d=3|\hK|/|\d\hK|$ is the radius of the insphere).
  Define
  \[
  \hG \bsig = \frac{\bsig}{d} (\hx - \hx_I).
  \]
  Then, setting $\hat{c}= \| \hG 1 \|_{\hdiv \hK} / |\d \hK|$ and
  $c=\hat{c} |\hK|^{\frac{1}{2}}$, we have
  \begin{align} \label{eq:Gdiv}
    \| \hG \bsig \|_{\hdiv \hK}^2
    & = \hat{c}^2
      \left|
      \int_{\partial {\hK}} \hsign\, ds \right|^2
      = \hat{c}^2 \left| \int_{{\hK}} \div \hsig \, dx \right|^2
      \leq c^2 \| \div \hsig \|_{\hK}^2
  \end{align}
  and $\hG \bsig = \bsig (\hx - \hx_I) \cdot \hat n / d = \bsig$, for all
  $\hx\in \d\hK$.
\end{proof}

\begin{lemma}
  \label{lem:E}
  There is a $c>0$ and a $\hE : Q_p(\d \hK) \to R_p(\hK)$ such that
  for any $\hsig$ in $\hdiv \hK$ with $\hsign \in Q_p(\d \hK)$, we
  have $\hat n \cdot (\hE \hsign)|_{\d\hK} = \hsign$,
  \[
  \| \hE( \hsign - \bsig) \|_{L^2(\hK)} \le c \|  \hsig \|_{\hdiv \hK },
  \quad \text{ and }\quad
  \dive (\hE (\hsign-\bsig)) = 0.
  \]
\end{lemma}
\begin{proof}
  We use the polynomial extension operator $ \extension $ from
  \cite[Theorem~7.1]{demkowicz-gopalakrishnan-polynomial-extension-iii}:
  Accordingly {\em (a)} if $\hq \in Q_p(\d \hK)$, then
  $\extension \hq$ is in $R_p(\hK),$ {\em (b)} if $\hq$ has zero mean,
  then $\dive(\extension \hq)=0$, and {\em (c)} if $\hat\tau$ is any
  extension of $\hq$ (i.e., $\hat \tau$ is a function in $\hdiv \hK$ satisfying
  $\hat n \cdot \hat\tau|_{\d \hK} = \hq$), then
  \[
  \| \extension \hq \|_{\hdiv \hK} \le c \| \hat\tau \|_{\hdiv \hK}.
  \]
  Since $\hsig$ is an extension of $\hsign$ and $\hG \bsig$ is an
  extension of $\bsig$,
  \begin{align*}
  \| \extension (\hsign -\bsig) \|_{\hdiv \hK}
    & \leq
      \| \extension \hsign \|_{\hdiv \hK}
    +
    \| \extension \bsig \|_{\hdiv \hK}
    \\
    & \le c \left( \| \hsig \|_{\hdiv \hK} + \| \hG \bsig \|_{\hdiv
      \hK}\right)
      \le c \| \hsig \|_{\hdiv \hK}.
  \end{align*}
  Finally, since $\hq = \hsign - \bsig$ has zero mean, we have
  $\dive ( \extension ( \hsign -\bsig) ) = 0.$
\end{proof}

Before the next result, recall that for any tetrahedral element $K,$
there is an affine homeomorphism $\Phi : \hK \to K$. Let $[D \Phi]$
denote the Jacobian matrix of its derivatives and let
$J = \det [D\Phi]$.  Define the Piola maps
\[
\Phi_* \hsig = J^{-1} [D \Phi_K] \hat \sigma \circ \Phi^{-1}, \qquad
\Phi^* \sigma  = J [D \Phi_K]^{-1} \sigma \circ \Phi.
\]
Clearly $\Phi_*$ maps functions on $\hK$ to $K,$ while $\Phi^*$ maps
in the opposite direction, from $K$ to $\hK$.
Letting $|K|$ denote the volume of the $K$ and
$h_K = \diam(K)$, we recall the following standard estimates
\cite{mixed-fem-book} for
affine $\Phi$: There is a $c>0$, depending only on shape regularity of
$K,$ such that
\begin{subequations}
  \label{eq:scaling}
  \begin{align}
    \|\Phi_* \hsig \|_{L^2(K)}^2
    &
      \le  c \, \frac{h_K^2}{|K|} \| \hsig \|_{L^2(\hK)}^2,
    &
      \| \dive (\Phi_* \hsig)\|_{L^2(K)}^2
    &
      \le  \frac{c}{|K|} \| \dive \hsig \|_{L^2(\hK)}^2,
    \\
    \|\Phi^* \sigma \|_{L^2(\hK)}^2
    &
      \le   c \,\frac{|K|}{h_K^2} \| \sigma \|_{L^2(K)}^2,
    &
      \| \dive (\Phi^* \sigma) \|_{L^2(\hK)}^2
    & \le  c \,{|K|}\; \| \dive \sigma \|_{L^2(K)}^2,
  \end{align}
\end{subequations}
for all $\hsig \in \hdiv \hK$ and $\sigma \in \hdiv K$.

We are now in a position to put everything together and prove the main
norm equivalence result of this section.

\begin{theorem}
  \label{thm:tet}
  If $\oh$ is shape regular, then there is a $c_3>0$ independent of
  $h$ and $p$ (and depending only on the shape regularity) such that
  \begin{equation}
    \label{eq:Q2Qh}
    \| q \|_Q \le \| q \|_{Q_h} \le c_3 \|q \|_Q
  \end{equation}
  for all $q \in Q_h$.
\end{theorem}
\begin{proof}
  The lower inequality follows from
  \[
  \| q\|_Q = \inf_{ \tau \in \trace^{-1}\{ q\} } \|\tau \|_{\Hdiv}
  \le
  \inf_{ \tau_h \in R_h \cap \trace^{-1}\{ q\} } \|\tau_h\|_{\Hdiv}  =
  \| q \|_{Q_h}.
  \]

  To prove the upper inequality,
  pick any $K \in \oh$, set
  \[
  \sigma = (Eq)|_K, \quad \hsig = \Phi^*\sigma,
  \quad   \FF_K q = \Phi_* \hF \hsign,
  \]
  where $\hG$ and $\hE$ are as given by
  Lemmas~\ref{lem:G} and \ref{lem:E}, $E$ is the minimal extension
  in~\eqref{eq:Eq}, and
  \[
  \hF \hsign = \hG \bsig  + \hE ( \hsign - \bsig).
  \]
  Clearly, $n \cdot (\FF_K q)|_{\d K} = q$ and the function $\FF q$,
  defined by $(\FF q)|_K = \FF_K q$ for all $K \in \oh$, is in
  $\Hdiv.$ Moreover the estimates of Lemmas~\ref{lem:G}
  and~\ref{lem:E}, together with~\eqref{eq:scaling}, imply
  \begin{align*}
    \| \FF q \|_{L^2(K)}^2
    &
      \le  c \, \frac{h_K^2}{|K|}
      \left( \| \hG \bsig \|_{L^2(\hK)}^2 + \| \hE (\hsign - \bsig)
      \|_{L^2(\hK)}^2
      \right)
     \le c \, \frac{h_K^2}{|K|} \| \hsig \|_{\hdiv \hK}^2
    \\
    &
      = c \, \frac{h_K^2}{|K|} \| \Phi^* \sigma \|_{\hdiv \hK}^2
      \le c \left(  \| \sigma \|_{L^2(K)}^2
      + h_K^2 \|\dive \sigma \|_{L^2(K)}^2\right),
    \\
    \|\dive \FF q \|_{L^2(K)}^2
    & \le \frac{c}{|K|} \| \dive (\hG \bsig) \|_{L^2(\hK)}^2
      \le  \frac{c}{|K|} \| \dive\hsig \|_{L^2(\hK)}^2
      =  \frac{c}{|K|} \| \dive(\Phi^* \sigma) \|_{L^2(\hK)}^2
      \\
    & \le  c \| \dive \sigma \|_{L^2(K)}^2.
  \end{align*}
  We have thus obtained,
  for any $q \in Q_h$, an extension $\FF q \in R_h$ satisfying
  \[
  \| \FF q \|_{\Hdiv} \le c \| \sigma \|_{\Hdiv} = c \| Eq \|_{\Hdiv}
  = c\| q \|_Q.
  \]
  Since $\| q \|_{Q_h}$ is the infimum of $\| \tau_h \|_{\Hdiv}$ over
  all $\tau_h \in R_h$ satisfying $\trace \tau_h = q$, the inequality
  $\| q\|_{Q_h} \le \| \FF q \|_{\Hdiv}$ holds and completes the
  proof.
\end{proof}

\subsection{General meshes}

We now briefly remark on how the norm equivalence of
Theorem~\ref{thm:tet} may be extended to more general elements and
meshes.  While a general theorem for all element
shapes is beyond the scope of this paper, we wish to provide pointers
on what arguments need extension. The proof of Theorem~\ref{thm:tet}
depends on three ingredients: {\em (a)} Lemma~\ref{lem:G}, {\em (b)}
Lemma~\ref{lem:E}, and {\em (c)} the scaling
estimates~\eqref{eq:scaling}.  Moving from tetrahedral to other
element shapes, we must first obtain generalizations of the extension
operators of Lemmas~\ref{lem:G} and~\ref{lem:E} on the reference
element~$\hK$ for the new shapes. We show how this can be done for two
other element shapes, one in two dimensions and another in three dimensions.

{\em Triangles:} The extension $\hG$ constructed in the proof of
Lemma~\ref{lem:G} continues to work for the unit triangle if we set
$\hx_I$ to be the center of the inscribed circle of the triangle. As for
Lemma~\ref{lem:E}, if $\hK$ is a triangle, then the extension
of~\cite[Corollary~2.2]{AinswDemko09} has all the properties stated in
the lemma.

{\em Cubes:} To obtain the result of Lemma~\ref{lem:G} when $\hK$ is
the unit cube, we set $\hx_I = (1/2, 1/2, 1/2)$ and
$\hG \bsig = 2 (\hx - \hx_I) \bsig$. Then proceeding as
in~\eqref{eq:Gdiv}, we obtain the result.  The extension operators
constructed in~\cite{CostaDaugeDemko08} for each $p$ provide the
required $\hE$ in Lemma~\ref{lem:E} when $\hK$ is a cube.

The scaling estimates~\eqref{eq:scaling} are valid for affine mappings
$\Phi$. We next comment on meshes with curved elements, which are images
of reference elements under a possibly nonlinear $\Phi$. If $\Phi$ is such
that the estimates of~\eqref{eq:scaling} with a properly (re)defined
$h_K$ and $|K|$ for curvilinear elements~$K$ hold, then the proof of
Theorem~\ref{thm:tet} can be generalized. Examples of nonlinear $\Phi$
where such geometrical quantities can be identified can be found
in~\cite{Berna89,mixed-fem-book}.

\section{An algebraic Schur complement result}  \label{sec:schur-result}

The purpose of this section, which can be read independently of
the rest of the paper, is to present a simple matrix result, whose
relevance to our problem will be clear in the next section. The result
is a generalization
of~\cite[Lemma~4.2]{brunner-kolev-amg-explicit}. Suppose
$i \cup f = \{1, 2, \ldots, m\}$ and $j \cup e = \{1, 2, \ldots, l\}$
are disjoint partitions of two index sets.  Let $\dof D$ be an
$m \times m$ symmetric positive definite matrix and $\dof H$ be an
$m \times l$ matrix (both with real entries).
We use standard block notations, e.g.,
$\dof{x}_f$ denotes the restriction of a vector $\dof{x}$ to
$f$-indices, and the matrices have block forms
\begin{equation}
  \label{eq:DandT}
  \dof{D} = \left[ \begin{array}{cc}
    \dof{D}_{ii}&  \dof{D}_{if} \\
    \dof{D}_{fi}&  \dof{D}_{f\!f}
  \end{array} \right], \quad
  \dof{H} = \left[ \begin{array}{cc}
    \dof{H}_{ij}&  \dof{H}_{ie} \\
    \dof{H}_{fj}&  \dof{H}_{fe}
  \end{array} \right].
\end{equation}
Define $\dof{S}$ to be the Schur complement
$ \dof{S} = \dof{D}_{f\!f} - \dof{D}_{fi} \dof{D}_{ii}^{-1}
\dof{D}_{if}.$
Let $\diag(\dD)$ denote diagonal matrix formed from the diagonal of
$\dD$.

\begin{lemma} \label{lem:alg_schur} %
  Suppose there is a $c_4>0$ such that every $ \du \in \RRR^m$ can
  be decomposed as $ \du = \dv + \dH \dr$, for some
  $\dv \in \RRR^m$ and $ \dr \in \RRR^l,$ such that
  \begin{equation}
    \label{eq:voldecomp}
    (\diag(\dD) \dv, \dv) +
    (\dD \dH \dr, \dof H \dr ) \le c_4 ( \dD \du, \du).
  \end{equation}
  Then for any $\du \in \RRR^m$ there exist $\dv \in \RRR^m$ and $\dr\in\RRR^l$
  (not necessarily the same as in the assumption), depending only on $\du_f$,
  such that the decomposition $\du_f = \dv_f + [\dH \dr]_f$ holds and satisfies
  \[
  (\diag(\dS) \dv_f, \dv_f) +
  ( \dS [\dH \dr]_f, [\dH \dr]_f)
  \le c_4 (\dS \du_f, \du_f).
  \]
\end{lemma}
\begin{proof}
Let $\extmat$ be the matrix representation of the extension operator $E_h$
\[
\extmat =
\begin{bmatrix}
 -\dof{D}_{ii}^{-1} \dof{D}_{if} \\
\dof{I}_{f\!f}
\end{bmatrix}.
\]
Since $ \dof{S} = \extmat^T \dof{D} \extmat$,
from the well-known properties of Schur complements
\begin{equation}
  (\dof{S} \dof{x}_f, \dof{x}_f)  =
  (\dof{D} \extmat \dof{x}_f, \extmat \dof{x}_f)
   = \inf_{  \{ \dy \in \RRR^m :\, \dy_f = \dx_f\} }
   (\dof{D} \dy, \dy) \le (\dD \dx, \dx),
  \qquad \forall \dx \in \RRR^m.
  \label{eq:schurmin2}
\end{equation}

Now, given any $\du \in\RRR^m$, let us set $\dw = \extmat \du_f$ and let
$\dv\in\RRR^m$, $\dr\in\RRR^l$ be such that $\dw = \dv + \dH\dr$
(and in particular $\du_f = \dw_f = \dv_f + [\dH\dr]_f$) and
\begin{equation}\label{eq:lemma41est1}
(\diag(\dD) \dv, \dv) +
( \dD \dH \dr, \dH \dr)
\le c_4 (\dD \dw, \dw).
\end{equation}
By~\eqref{eq:schurmin2}, with $\dx=\dH\dr$
\begin{equation}
  \label{eq:3}
  ( \dS [\dH \dr]_f, [\dH \dr]_f)
  \le
  ( \dD \dH \dr, \dH \dr).
\end{equation}
Next, consider the $k$-th diagonal entry of $\dS$ which can be expressed as
$\dS_{kk} = (\dS \dof e_k, \dof e_k)$ where $\dof e_k$ is the vector with
entries $[\dof e_k]_s = \delta_{ks}$. Setting $\dx^T=[ \dof{0}^T \; \dof{e}_k^T ]$
in \eqref{eq:schurmin2}, we get
\[
\dS_{kk} = (\dS \dof e_k, \dof e_k) \le (\dD_{f\!f} \dof e_k, \dof e_k ) =
[\dD_{f\!f}]_{kk}\,.
\]
Since all diagonal entries of $\dD$ are positive, we conclude that
\begin{equation}\label{eq:lemma41est2}
(\diag(\dS) \dv_f, \dv_f) \le
(\diag(\dD_{f\!f}) \dv_f, \dv_f) \le
(\diag(\dD) \dv, \dv).
\end{equation}
Adding the estimates \eqref{eq:3} and \eqref{eq:lemma41est2} and then using
\eqref{eq:lemma41est1} we arrive at
\[
(\diag(\dS) \dv_f, \dv_f) + ( \dS [\dH \dr]_f, [\dH \dr]_f) \le
(\diag(\dD) \dv, \dv) + ( \dD \dH \dr, \dH \dr) \le
c_4 (\dD \dw, \dw).
\]
Noting that $(\dD \dw, \dw) = (\dD \extmat \du_f, \extmat \du_f) =
(\dS \du_f, \du_f)$ completes the proof.
\end{proof}

The statement of Lemma \ref{lem:alg_schur} can be easily extended to the case
of more than one matrix $\dH$: assume that we have a sequence of real matrices
$\dH_k$ with dimensions $m\times l_k$, $k=1,\ldots,n$.
\begin{corollary}\label{corrollary-decomp}
Suppose there is $c_4>0$ such that for all $\du\in\RRR^m$ there exist
$\dv\in\RRR^m$ and $\dr_k \in\RRR^{l_k}$, $k=1,\ldots,n$, such that
\[
  \du = \dv + \sum_{k=1}^n \dH_k \dr_k \,,
    \qquad\text{and}\qquad
  (\diag(\dD) \dv, \dv) + \sum_{k=1}^n (\dD \dH_k \dr_k, \dH_k \dr_k) \le
  c_4 (\dD \du, \du).
\]
Then for any $\du \in \RRR^m$ there exist $\dv \in \RRR^m$ and
$\dr_k\in\RRR^{l_k}$, $k=1,\ldots,n$ (not necessarily the same as in the
assumption), depending only on $\du_f$, such that
\[
  \du_f = \dv_f + \sum_{k=1}^n [\dH_k \dr_k]_f \,,
  \quad\text{and}\quad
  (\diag(\dS) \dv_f, \dv_f) +
  \sum_{k=1}^n (\dS [\dH_k \dr_k]_f, [\dH_k \dr_k]_f) \le
  c_4 (\dS \du_f, \du_f).
\]
\end{corollary}

\section{Preconditioning the $Q_h$-norm using
an interface decomposition} 
 \label{sec:ads}

In Section~\ref{sec:qnorm}, we reduced the problem of preconditioning $\|\cdot\|_Q^2$ to that
of preconditioning $\|\cdot\|_{Q_h}^2$. In this section we propose
a scalable method for preconditioning the $Q_h$-norm,
by further reducing the problem to
that of preconditioning the Gram matrix of the $ \Hdiv $
inner product. Such matrices can be efficiently handled by recent algebraic multigrid techniques
\cite{kolev-vassilevski-ads}, resulting ultimately in a good preconditioner for the DPG matrix
$\dof{A}$, as shown in the next section.

Let $\{ r_m \}$ denote a finite element basis of $R_h$. Define
$\dof{D}$ to be the Gram matrix of the $ \Hdiv $ inner product in the
$ \{ r_m\} $ basis.  We partition the degrees of freedom of
$ \{ r_m\} $ into those associated with the interior of elements --
denoted by $i$ --
and those on the element interfaces
-- denoted by $f$ -- and block partition $\dD$
as in~\eqref{eq:DandT}.  Recall the notational conventions from
Section~\ref{sec:dpg} that allow us to move from functions $q$ to their
vector representations $\dof q$ using appropriate basis expansions.
As already noted
in~\eqref{eq:schurmin2}, the Schur complement
$ \dof{S} = \dof{D}_{f\!f} - \dof{D}_{fi} \dof{D}_{ii}^{-1}
\dof{D}_{if} $ satisfies
\begin{equation} \label{eq:S}
  (\dof{S} \dof{q}, \dof{q}) =
\inf_{ \{ r \in R_h : \; \dof{r}_f = \dof{q} \} } (\dof{D} \dof{r}, \dof{r})
= \| E_h q \|^2_{\Hdiv}
= \| q \|^2_{\interfacespace},
\end{equation}
i.e., to precondition the $Q_h$-norm we need to construct
a good preconditioner for $ \dof{S} $.

The characterization of the $Q_h$-norm in terms of an $ H(\div) $-norm suggests the use
of an $ H(\div)$ preconditioner. Indeed, if $\dT =
\begin{bmatrix}
  \dof{0}_{fi} & \dI_{f\!f}
\end{bmatrix}$
denotes the restriction operator such that $ \dT \dr = \dr_f$ for all
$r \in R_h,$ then it follows from
\[
\dD^{-1} =
\begin{bmatrix}
  \dI & -{\dD}_{ii}^{-1} {\dD}_{if} \\
  0 & \dI
\end{bmatrix}
\begin{bmatrix}
  {\dD}_{ii}^{-1} & 0 \\
  0 & {\dS}^{-1}
\end{bmatrix}
\begin{bmatrix}
  \dI & 0 \\
  -{\dD}_{fi}{\dD}_{ii}^{-1} & \dI
\end{bmatrix}
\]
that $\dS^{-1} = [\dD^{-1}]_{f\!f} = \dT \dD^{-1} \dT^T$.
Thus, replacing $\dD^{-1}$ by any spectrally equivalent
$\Hdiv$-preconditioner will give us a spectrally equivalent
preconditioner for $\dS$.
In particular, we may use the
Auxiliary-space Divergence Solver (ADS) of
\cite{kolev-vassilevski-ads}.

It is well known that ADS is a good preconditioner for many problems
set in the $\Hdiv$-conforming space $ \rt$.  However, we want to
precondition the interface operator $\dS$ using only the interface
degrees of freedom.  The ADS preconditioner when applied to $R_h$ uses
all degrees of freedom of $R_h$, and not merely the interface degrees
of freedom in $Q_h$. This can become a significant addition to the
cost as the order $p$ increases.

What can we expect when the algebraic ADS is directly applied
to the interface space $Q_h$? To answer this, we examine below the
stable decomposition underpinning the theory of ADS and employ
Corollary~\ref{corrollary-decomp} to get an analogous stable decomposition
restricted to the interface.
For simplicity, we now focus  on the three-dimensional case.
(The two-dimensional case is similar once curl is properly defined.)
Let $N_h$ denote the $H(\curl)$-conforming Nedelec space of the
first kind on the same mesh, which is in correspondence with $R_h$ in
the standard finite element exact sequence.

The additive variant of ADS provides a preconditioner for
$\dof D$ in the form
\begin{equation} \label{adsprec}
  \adsprec = \dof{R} +
  \projmatrix \,\dof{B}^{\projmatrix}\, \projmatrix^T +
  \curlmatrix\, \dof{B}^{\curlmatrix}\, \curlmatrix^T
\end{equation}
where the ingredients are as follows:
\begin{enumerate}
\item $\dof{R}$ is a simple smoother for the global matrix $\dof{D}$,
  for example,
one symmetrized  Gauss-Seidel iteration.
\item $ \projmatrix $ is the matrix representation of the
  Raviart-Thomas interpolation operator
  from $\trialspace \times \trialspace \times \trialspace$ (or simply $U_h^3$)
  to $\rt$ obtained using a standard  basis $\{ u_l \}$ of
  $\trialspace$ and the basis $\{ r_m \}$ of $\rt$.
\item $ \curlmatrix$ is the matrix representation of $ \curl: N_h \to  R_h $
  using a standard  basis $\{ n_k \}$ of $\nedelec$ and the basis $\{ r_m \}$
  of $\rt$.
\item $ \dof{B}^{\projmatrix} $ is a standard algebraic $ H^1 $ solver, for
  example BoomerAMG from \cite{henson-yang-boomeramg, scaling_2012}, applied to
  the matrix $\projmatrix^T \dof{D} \projmatrix$.
\item $ \dof{B}^{\curlmatrix} $ is an algebraic Maxwell solver, such
  as the auxiliary space Maxwell solver of
  \cite{kolev-vassilevski-ams} applied to $\curlmatrix^T \dof{D}
  \curlmatrix$.
\end{enumerate}

Just as we partitioned the degrees of freedom of $R_h$ into interior
($i$) and interface ($f$) ones, we can partition the degrees of
freedom of $U_h^3$ into its interior $\ju$ and its interface $(\eu)$
degrees of freedom. Similarly the degrees of freedom of $N_h$ are
partitioned into sets $\jn$ (interior) and $\en$ (interface).  An
important property of the matrices $\projmatrix$ and $\curlmatrix$ is
that when we decompose them into the interior and interface degrees of
freedom, their block form is
\begin{equation}
  \projmatrix = \left[ \begin{array}{cc} \projmatrix_{i\ju} & \projmatrix_{i\!\eu} \\
                         \dof 0& \projmatrix_{f\!\eu} \end{array} \right], \quad
  \curlmatrix = \left[ \begin{array}{cc} \curlmatrix_{i\jn} & \curlmatrix_{i\!\en} \\
                         \dof 0& \curlmatrix_{f\!\en} \end{array} \right].
  \label{eq:boundaryop}
\end{equation}
The fact that $\dof{\Pi}_{f\ju}$ and $\dof{C}_{f\jn}$
are zero blocks follows from the definition of
the finite element spaces $ \lagrange, \rt $,
$ \nedelec$ and their degrees of freedom,
 e.g., the $R_h$ degrees of freedom on a face
for the curl of a function in $N_h$ depend only on the $N_h$
degrees of freedom associated with that face.

The rationale behind the preconditioner construction
in~\eqref{adsprec} comes from the theory of auxiliary space
preconditioners~\cite{HiptmXu07}. For example,
it is possible to prove \cite[Section 5.2]{kolev-vassilevski-ads}
under further simplifying assumptions 
that
any $ u \in \rt $ can be decomposed into
\begin{subequations}
  \label{eq:decompose}
\begin{equation}
  \dof{u} = \dof{v} + \projmatrix \dof{z} + \curlmatrix \dof{y}
  \label{eq:matrix-hx}
\end{equation}
with 
$z \in \lagrange \times \lagrange \times \lagrange$,
$y \in \nedelec, $ and $ v \in \rt $ such that
\begin{equation}
  (\dof{diag}(\dof{D}) \dof{v}, \dof{v}) + (\dof{D} \projmatrix \dof{z}, \projmatrix \dof{z}) +
  (\dof{D} \curlmatrix \dof{y}, \curlmatrix \dof{y}) \leq c_5 (\dof{D} \dof{u}, \dof{u})
  \label{eq:hx-estimate}
\end{equation}
\end{subequations}
where
$ c_5>0 $ is a constant
independent of the size of the problem.
This is enough to conclude~\cite{Nepom07} that
 $\dof{B}$ is a good preconditioner
for $\dof{D}^{-1}$ (and the ``goodness'' is measured by $c_5$
as the condition number
of the preconditioned system increases  with $c_5$).
In practice, $\dof{B}$ often serves a good preconditioner
for
$\dof{D}^{-1}$ even when a rigorous proof of~\eqref{eq:decompose}
is  difficult (such as for non-conforming
irregular meshes and  discontinuous  material coefficients).
Loosely speaking,~\eqref{eq:decompose}
means that $\du$ can be decomposed
into well-behaved components
in the ranges of $\projmatrix$ and $\curlmatrix$ with a
small remainder $\dv.$


When a purely algebraic implementation of ADS is applied to $\dS$, it
results in the preconditioner
\begin{equation}
  \label{eq:Bf}
  \dB^f = \dR^f + \dPi_{f\!\eu} \dB^\projmatrix_{\eu\!\eu} \dPi_{f\!\eu}^T
+ \dC_{f\!\en} \dB^\dC_{\en\!\en} \dC_{f\!\en}^T
\end{equation}
which {\em uses only the interface degrees of freedom} of all the spaces
involved. Here $\dR^f$ is a simple point smoother, like the
symmetrized Gauss-Seidel iteration, applied to~$\dS$.  Just
as~\eqref{eq:decompose} implies that $\dB$ is a good preconditioner
for $\dD$, a stable interface decomposition is required for $\dB^f$
to be a good preconditioner for $\dS$. We will now show that the
decomposition~\eqref{eq:decompose} implies a stable interface decomposition.

\begin{lemma}     \label{lemma:reducedads} %
  If~\eqref{eq:decompose} holds,
  then
  any $q \in Q_h$ can be decomposed as
  \[
    \dof{q} = \dv_f + \projmatrix_{f\!\eu} \dof{z}_f + \curlmatrix_{f\!\en} \dof{y}_f
  \]
  where $v \in R_h$, $z \in U_h^3$ and $y \in N_h$ and their interface
  degrees of freedom satisfy
  \[
  (\diag(\dS) \dv_f, \dv_f) +
  (\dS \projmatrix_{f\!\eu} \dof{z}_{\eu}, \projmatrix_{f\!\eu} \dof{z}_{\eu}) +
  (\dof{S} \curlmatrix_{f\!\en} \dof{y}_{\en}, \curlmatrix_{f\!\en} \dof{y}_{\en})
  \leq c_5 (\dof{S} \dof{q}, \dof{q}).
  \]
\end{lemma}
\begin{proof}
Apply Corollary~\ref{corrollary-decomp} with $\dH_1 = \dPi$, and $\dH_2 = \dC$,
and observe that
$[\dPi \dof{z}]_f = \projmatrix_{f\!\eu} \dof{z}_{\eu}$ and
$[\dC \dof{y}]_f = \curlmatrix_{f\!\en} \dof{y}_{\en}$
due to~\eqref{eq:boundaryop}.
\end{proof}

Informally, the result of the lemma can be stated as follows: if ADS works for the matrix $\dof{D}$
(a volumetric discretization of $\| \cdot \|_{\Hdiv}$), it will also work for its Schur complement $\dof{S}$ (an
interfacial discretization of $\|\cdot\|_{Q_h}$). Since we assume the former, we can conclude that
ADS will be an effective preconditioner for $ \dof{S} $.

\section{Scalable preconditioner} \label{sec:prec}

We are now ready to put all the pieces together and define a scalable preconditioner for the
original DPG matrix $\dof{A}$.
Our basic premise
is that {\em (i)}~the algebraic ADS is a good solver for the
Gram matrix of the $\Hdiv$-inner product in $R_h$,
in the sense that~\eqref{eq:decompose}
holds, and {\em (ii)}~the algebraic solver
BoomerAMG~\cite{henson-yang-boomeramg}, denoted by
$\dB^o$, is a good preconditioner for
the
Gram matrix $\dof G$ of the
 $H^1(\om)$-inner product on $U_h$,
in the sense that the spectral condition
number is bounded independent of discretization size $h$ and polynomial order $p$, that is,
\begin{equation}
  \label{eq:amg_ass}
  \kappa(\dB^o \dof G) \le c_6.
\end{equation}
Combining this with the $\dB^f$ defined in~\eqref{eq:Bf}, we have the
following result.

\begin{theorem}
  Assume that~\eqref{eq:q-equivalence}, \eqref{eq:Q2Qh},
  \eqref{eq:decompose} and~\eqref{eq:amg_ass} hold. Then
  the block-diagonal matrix
\begin{equation}
  \left[ \begin{array}{cc}
    \dB^o& \\
    &  \dB^f
  \end{array} \right] \label{eq:ideal-preconditioner}
\end{equation}
is a preconditioner for \dof{A} and the condition number of the
preconditioned system depends only on $c_1, c_2, c_3, c_5$ and $c_6$.
\end{theorem}
\begin{proof}
From \eqref{eq:q-equivalence},
for any $x = (u, q) \in X$,
we have
$$
c_1 (\|u\|_{H^1(\om)}^2 + \|q\|_{Q}^2) \leq (\dof{A} \dof{x}, \dof{x}) \leq
c_2 (\|u\|_{H^1(\om)}^2 + \|q\|_{Q}^2).
$$
Using~\eqref{eq:Q2Qh},
\begin{equation}
  \label{eq:2}
c_1 \|u\|_{H^1(\om)}^2 + c_1c_3^{-2}\|q\|_{Q_h}^2
\leq (\dof{A} \dof{x}, \dof{x})
\le
c_2 \|u\|_{H^1(\om)}^2 + c_2 \|q\|_{Q_h}^2
\end{equation}
Hence the result follows from Lemma~\ref{lemma:reducedads} and~\eqref{eq:amg_ass}.
\end{proof}

In practice, the application of $\dB^o$ and $\dB^f$ requires the
availability of the Gram matrices $\dof G$ and $\dD$, which may be inconvenient.
What we have in hand is $\dof A$. Hence instead of
the preconditioner in~\eqref{eq:ideal-preconditioner}, we may
use the block preconditioner
\begin{equation*}
  \begin{bmatrix}
    \dP^o & \\
        & \dP^f
  \end{bmatrix}
\end{equation*}
where $\dP^o$ and $\dP^f$ are the
algebraic solvers BoomerAMG and ADS
applied directly to the principal minors of $\dof{A}$ corresponding to
$U_h$ and $Q_h$, namely to
$ \dof{A_0} = \dof{B}_0^T \dof{M}^{-1} \dof{B_0} $ and
$ \dof{A_1} = \dof{B}_1^T \dof{M}^{-1} \dof{B}_1$ respectively.
The justification for this comes from the observation
that by taking $q=0$ in~\eqref{eq:2}, we can conclude that
$$
c_1 \|u\|_{H^1(\om)}^2 \leq (\dof{A_0} \dof{u}, \dof{u}) \leq c_2 \|u\|_{H^1(\om)}^2,
$$
i.e., $ \dof{A_0} $ is spectrally equivalent to $\dof G$.
Similarly, by taking $u=0$  in~\eqref{eq:2}, we have
$$
c_1 c_3^{-2}\|q\|_{Q_h}^2 \leq (\dof{A_1} \dof{q}, \dof{q}) \leq c_2 \|q\|_{Q_h}^2.
$$
Thus instead of
preconditioning the  matrices $\dof G$
and $\dS$, whose quadratic forms give
the norms $\| \cdot \|_{H^1(\om)}^2$
and $\| \cdot \|_{Q_h}^2$ respectively,
we can
directly precondition their spectrally equivalent
principal minors $\dof{A}_0$ and $\dof{A}_1$.
In our implementation it is in fact straightforward to construct the Gram matrix $\dof G$,
and we do so in order to build the AMG preconditioner $\dB^o$, but we use the principal minor
$ \dof{A_1} $ to construct the ADS preconditioner $ \dP^f $, so that the preconditioner we
use in the numerical results below takes the form
\begin{equation}
  \label{eq:practical-preconditioner}
  \begin{bmatrix}
    \dB^o & \\
        & \dP^f
  \end{bmatrix}.
\end{equation}

\section{Numerical results} \label{sec:results}

In this section we report some numerical results with the proposed DPG
preconditioner that test its performance with respect to the mesh
size $h$, the polynomial order of the trial space $p$, as well as the
orders of the test and interfacial spaces. We also examine the parallel
scalability of the new algorithm and examine its behavior on more
challenging problems with unstructured meshes and large coefficient jumps.

We apply a Conjugate Gradients (CG) solver to the problem~\eqref{eq:dpgmat}
preconditioned with the preconditioner~\eqref{eq:practical-preconditioner}
where $\dB^o$ and $\dP^f$ use a single V-cycle of BoomerAMG and ADS respectively.
The CG relative tolerance we used was $10^{-6}$.

Our implementation is freely available in the MFEM finite element
library \cite{mfem} and we used a slightly modified version of MFEM's
parallel Example 8 (version 3.2) to perform the numerical experiments in this
section. Specific ADS and BoomerAMG parameters and additional details can
be found in the source code of that example.

\subsection{Scalability with respect to $h,p$ for structured mesh}
\label{sec:structured}

Here we solve the test problem \eqref{eq:poisson} on the domain
$ \domain = (0,1)^3 $ with constant coefficient $ \coeff = 1 $ meshed with a uniform hexahedral grid.
The right hand side $ f $ is set to the constant one and zero Dirichlet boundary
conditions are imposed on all of $ \partial \domain $.

Table \ref{tab:np16iterations} reports results for experiments with varying mesh size $h$ (reported as number of finite elements) and polynomial orders $p.$ The order $p$
sets the polynomial degree of $ \trialspace $ to $ p $,
the order of $ \interfacespace $ to $ p - 1 $, and the order of $ \testspace $ to $p + d - 1 $ where $ d = 3 $ is the spatial dimension of $ \domain $. As mentioned in Section~\ref{sec:dpg}, Assumption~\eqref{eq:q-equivalence} holds in this setting.
The table reports iteration counts as well as the average reduction factor in the PCG iteration.
We observe that both of these convergence metrics are quite stable with respect to $h$ and $p$.

\begin{table}
\caption{Number of CG iterations and average reduction factors per iteration (in parenthesis) for various $h$ and $p$ refinement levels.}
\label{tab:np16iterations}
\begin{center}\begin{tabular}{llllll}
\toprule
  &  \multicolumn{5}{c}{order ($p$)} \\
  \cmidrule{2-6}
elements& 1& 2& 4& 6& 8 \\
\midrule
64    & \zz5 (0.06)&  \zz8 (0.14)&  12 (0.30)&  13 (0.34)&  13 (0.34) \\
512   & \zz7 (0.12)&  10 (0.23)&  12 (0.31)&  14 (0.36)& \zzz--- \\
4096  & \zz8 (0.18)&  10 (0.25)&  13 (0.33)& \zzz---& \zzz--- \\
32768 &   10 (0.22)&  10 (0.24)& \zzz---& \zzz---& \zzz--- \\
262144&   10 (0.22)&  \zzz---& \zzz---& \zzz---& \zzz--- \\
\bottomrule
\end{tabular}\end{center}
\end{table}



In Table \ref{tab:weak-scaling}, we explore the parallel scalability
of this algorithm, doing a weak scaling study where the number of
elements is kept constant per processor as we increase the number of
processors.  This particular test uses a trial space order of $ p=1 $
but a test space order of 2 rather than the theoretically necessary 3
(see the remarks in Section \ref{sec:reduced-order}).  The test was
run on an IBM BlueGene/Q machine, where we use four MPI tasks per
node.

While the number of iterations in Table~\ref{tab:weak-scaling} exhibits some growth,
the overall performance is reasonably scalable, and we are continuing to work
on improving the per-iteration run time in our implementation.

\begin{table}
\caption{Weak scaling for the solver with polynomial order fixed at $p=1$.}
\label{tab:weak-scaling}
\begin{center}\begin{tabular}{lccccc}
\toprule
  processors&  elements&  iterations&  conv. factor&  solve time  & time/iteration\\
\midrule
     4&  2.62e+5& \zz9&   0.21&   249.42s & 27.7s\\
    32&  2.10e+6&   11&   0.26&   473.84s & 43.1s\\
   256&  1.68e+7&   12&   0.29&   547.95s & 45.7s\\
  2048&  1.34e+8&   13&   0.32&   665.81s & 51.2s\\
 16384&  1.07e+9&   14&   0.37&   745.69s & 53.3s\\
\bottomrule
\end{tabular}\end{center}
\end{table}

\subsection{Influence of the order of the test space}
\label{sec:reduced-order}

Currently known theoretical results on the DPG method requires one to
set the test space a few degrees higher than the trial space.
Higher order test spaces can  significantly add to
the size of the discrete system \eqref{eq:dpgmat} and the overall
computational cost.
Our numerical results indicate that
test spaces of one degree lower than the theoretical requirement
often continue to yield a scalable method.
We have observed this for triangles, quadrilaterals, tetrahedra, and hexahedra.
In Table~\ref{tab:reduced-order}, we  present some representative results for the interesting case of triangles in two dimensions, where the scalability depends on the parity of $p$.
For even $p=4$, scalability requires a test space order one degree higher than for the odd order $p=5$. The dependence
of the error convergence rate  on
the parity of $p$ was discussed
in \cite{BoumaGopalHarb14}. It is interesting
to observe that our preconditioner also exhibits such dependence.

\begin{table}
\caption{Effect of test space order $r$ on scalability with respect to $h$-refinement for a triangular mesh.}
\label{tab:reduced-order}
\begin{center}\begin{tabular}{lccccr}
  \toprule
  & \multicolumn{2}{c}{$p=4$}& & \multicolumn{2}{c}{$p=5$} \\
    \cmidrule{2-3} \cmidrule{5-6}
  refine&  $r=4$&  $r=5$&  &  $r=5$&  $r=6$ \\
  \midrule
  1&  14&  12&  &   18&    18 \\
  2&  20&  14&  &   16&    16 \\
  3&  28&  15&  &   17&    17 \\
  4&  46&  15&  &   16&    16 \\
  5&  50&  16&  &   16&    16 \\
  6&  67&  17&  &   15&    15 \\
  \bottomrule
\end{tabular}\end{center}
\end{table}

\subsection{Scalability with respect to $h$ on unstructured meshes}
\label{sec:unstructured}

\begin{figure}
\begin{center}
  \includegraphics[width=0.32\textwidth]{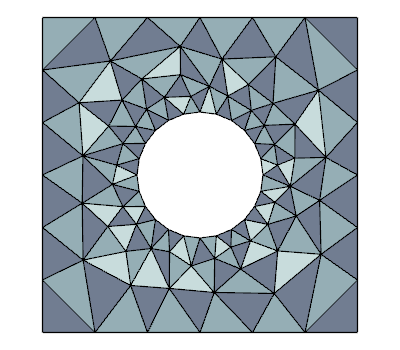}
  \includegraphics[width=0.32\textwidth]{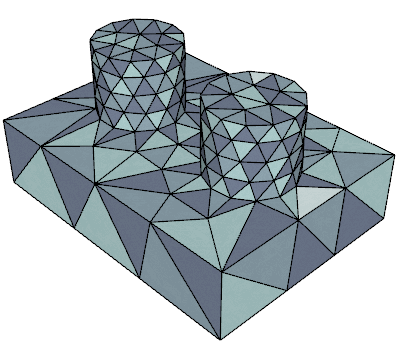}
  \includegraphics[width=0.32\textwidth]{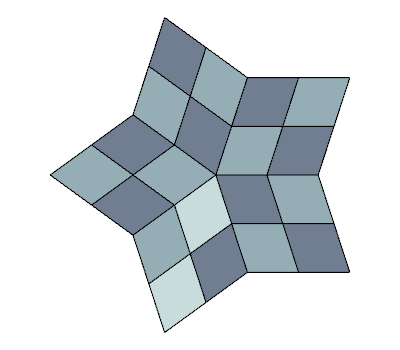}
\end{center}
\caption{Three unstructured meshes, using triangles, tetrahedra, and quadrilaterals.}
\label{fig:unstructured}
\end{figure}

Next we consider problems with different meshes, including unstructured triangular, quadrilateral, and tetrahedral meshes, using in particular the meshes shown in Figure \ref{fig:unstructured} at various levels of refinement.
The problem is the same as in Section~\ref{sec:structured} except for the mesh. We fix $p=1$ and focus on the scalability with respect to $h$.
The convergence results in Table~\ref{tab:un-np64iterations} demonstrate that the preconditioner continues to be scalable in these more general settings.

\begin{table}
\caption{Number of CG iterations and average reduction factors per iteration (in parenthesis) for several unstructured meshes at various refinement levels.}
\label{tab:un-np64iterations}
\begin{center}\begin{tabular}{lccc}
\toprule
refine& triangles& tetrahedra& quadrilaterals \\
\midrule
0&  13 (0.34)&\zz8 (0.17)& \zz9 (0.21) \\
1&  14 (0.37)&  11 (0.27)& 12 (0.31) \\
2&  14 (0.36)&  13 (0.35)& 13 (0.32) \\
3&  14 (0.35)&  15 (0.39)& 13 (0.33) \\
4&  14 (0.36)&  16 (0.42)& 12 (0.31) \\
5&  14 (0.37)&           & 12 (0.30) \\
6&  15 (0.38)&           & 12 (0.30) \\
7&  15 (0.38)&           & 12 (0.29) \\
8&  15 (0.39)&           & 12 (0.30) \\
\bottomrule
\end{tabular}\end{center}
\end{table}



\subsection{Behavior of solver with respect to contrast in coefficient}

In the following numerical results the coefficient $ \coeff $ in \eqref{eq:poisson} is piecewise constant, chosen randomly on each element, so that it is 1 with probability 1/2 and $ \coeff_0 $ with probability 1/2, where $ \coeff_0 $ is a specified constant across the mesh.
Here the $\dB^o$ component in \eqref{eq:practical-preconditioner} is constructed from an $H^1$ matrix assembled from the bilinear form in \eqref{eq:poisson} using the varying coefficient $\coeff$, and $\dP^f$ is constructed as usual using the principal minor of $\dof{A}$, which also includes the coefficient $\coeff$.
In Table \ref{tab:contrast-iterations} we report the number of iterations and average reduction factor for several refinement levels and choice of contrast $ \coeff_0 $.
The results show that the problem gets harder for high-contrast coefficients, as expected, but the solver still performs well.

\begin{table}
\caption{Number of CG iterations and average reduction factors per iteration (in parenthesis) for various values of the contrasts coefficient $ \coeff_0 $.}
\label{tab:contrast-iterations}
\begin{center}\begin{tabular}{lllllll}
\toprule
  &  \multicolumn{6}{c}{contrast} \\
  \cmidrule{2-7}
elements& 1e-06& 1e-04& 1e-02& 1e+00& 1e+02& 1e+04 \\
\midrule
   64&  13 (0.29)& 12 (0.31)& 10 (0.24)& 5 (0.06)&  8 (0.15)& 12 (0.24) \\
  512&  31 (0.64)& 29 (0.61)& 14 (0.36)& 7 (0.10)& 11 (0.27)& 17 (0.44) \\
 4096&  64 (0.80)& 49 (0.75)& 15 (0.39)& 8 (0.16)& 13 (0.33)& 33 (0.64) \\
32768& 119 (0.89)& 73 (0.83)& 16 (0.41)& 9 (0.20)& 14 (0.36)& 43 (0.72) \\
\bottomrule
\end{tabular}\end{center}
\end{table}


\section{Conclusions} \label{sec:conclusions}
In this paper we presented a scalable preconditioner for the primal DPG
formulation of the Poisson problem based on parallel algebraic multigrid
techniques.
We proved that the preconditioner is optimal under certain assumptions on
the mesh and problem coefficients.
We also demonstrated that the new algorithm performs well on a wide variety of
problems, including some where the theory is not applicable.
Due to its algebraic nature, the preconditioner is easy to apply in practice, and has a
freely available implementation in the MFEM library.


\bibliographystyle{siam}
\bibliography{dpg}

\end{document}